\documentclass[11pt]{article}
\usepackage{graphicx,color}
\usepackage{amsmath, amsthm, amssymb}
\usepackage[labelsep=endash]{caption}

\usepackage{graphicx}
\usepackage{latexsym,amssymb}
\usepackage{amsthm}
\usepackage{indentfirst}
\usepackage{amsmath}
\usepackage{color}
\usepackage{xcolor}
\usepackage[all]{xy}
\usepackage{hyperref}
\newtheorem{Theorem}{Theorem}[section]
\newtheorem*{Theorem A}{Theorem A}
\newtheorem*{Theorem B}{Theorem B}
\newtheorem*{Theorem C}{Corollary C}
\newtheorem*{Theorem D}{Corollary D}

\newtheorem*{Conj*}{Conjecture}

\newtheorem{Definition}[Theorem]{Definition}

\newtheorem{Lemma}[Theorem]{Lemma}

\newtheorem{Remark}{Remark}

\newtheorem{Remark-numbered}[Theorem]{Remark}
\newtheorem{Remarks-numbered}[Theorem]{Remarks}
\newtheorem{Corollary}[Theorem]{Corollary}

\newtheorem*{Theorem B'}{Theorem B'}
\newtheorem{Claim-numbered}{Claim}

\def\triangleq{\stackrel{\triangle}{=}}

\def\dim{\operatorname{dim}}

\begin{document}
\title{Centralizer of fixed point free separating flows}
\author {Bo Han and Xiao Wen\footnote{Xiao Wen was partially supported by National Key R\&D Program of China (No.2022YFA1005801), National Natural Science
Foundation of China (No.12071018) and the Fundamental
Research Funds for the Central Universities.}}

\maketitle

\begin{abstract}
In this paper, we study the centralizer of a separating continuous flow without fixed points. We show that if $M$ is a compact metric space and $\phi_t:M\to M$ is a separating flow without fixed points, then $\phi_t$ has a quasi-trivial centralizer, that is, if a continuous flow $\psi_t$ commutes with $\phi_t$, then there exists a continuous function $A: M\to\mathbb{R}$ which is invariant along the orbit of $\phi_t$ such that $\psi_t(x)=\phi_{A(x)t}(x)$ holds for all $x\in M$. We also show that if $M$ is a compact Riemannian manifold without boundary and $\Phi_u$ is a separating $C^1$ $\mathbb{R}^d$-action on $M$, then $\Phi_u$ has a quasi-trivial centralizer, that is, if $\Psi_u$ is a $\mathbb{R}^d$-action on $M$ commuting with $\Phi_u$, then there is a continuous map $A: M\to\mathcal{M}_{d\times d}(\mathbb{R})$ which is invariant along orbit of $\Phi_u$ such that $\Psi_{u}(x)=\Phi_{A(x)u}(x)$ for all $x\in M$. These improve Theorem 1 of \cite{O} and Theorem 2 of \cite{BRV} respectively.

\bigskip

{\noindent}Keyword: Quasi-trivial centralizer, Separating flow, fixed point free flow, homogenous $\mathbb{R}^d$-action

\end{abstract}

\section{Introduction}
Expansiveness and centralizers are important notions in dynamical systems. In 1970, Walters firstly related these two concepts by proving that expansive homeomorphisms have unstable centralizers. For flows, the story is more complicated because of the choice of definitions of expansiveness for flows. Kato and Morimoto \cite{KM} proved that Anosov flow has quasi trivial centralizer, they firstly used the tool called expansive in their proof, which from Walters and Bowen. Then in 1976, Oka \cite{O} extended the above conclusion to expansive flows. Here the definition of expansiveness is from Bowen and Walters. Note that expansiveness of Bowen and Walters is relatively strict, and fixed points are isolated under this type of expansiveness. Recently, in 2018, Bonomo-Rocha-Varandas \cite{BRV} studied the centralizer for Komuro expansive flows, they improved the result of Oka by proving that the centralizer of every $C^\infty$ Komuro-expansive flow with non-resonant hyperbolic singularities is trivial. In another paper, Lennard Bakker, Todd Fisher and Boris Hasselblatt \cite{BFH} improved the result of Oka by showing that kinematic-expansivity implies that flows have quasi-trivial centralizer. For the $C^1$ flows, Martin Leguil, Davi Obata and Bruno Santiago \cite{LOS} proved that $C^1$ separating flow has quasi trivial centralizer if the singularities of flow are all hyperbolic. In this paper, we will consider the improvement of the above results by considering $C^0$ separating flows. Note that separating property is a substantial weakening of both kinematic expansiveness and Komuro expansiveness.

Let $M$ be a compact metric space with metric $d$, $\phi:\mathbb{R}\times M\to M$ be a continuous flow without fixed points. Let $\phi_t(\cdot)=\phi(t,\cdot)$. The {\it orbit} of point $x$ under the action of $\phi$ is denoted by $Orb(x, \phi)=\{\phi_t(x)|t\in\mathbb{R}\}$. If $Orb(x, \phi)=\{x\}$ then we say $x$ is a {\it fixed point} of $\phi_t$. If there is $T>0$ such that $\phi_T(x)=x$ and $x$ is not a fixed point, then  we say that $x$ is a {\it periodic point} of $\phi_t$ and the orbit of $x$ is a {\it periodic orbit}.

\begin{Definition}
Let $M$ be a compact metric space with metric $d$, $\phi:\mathbb{R}\times M\to M$ be a continuous flow. We say that $\phi$ is separating if there exists $\delta>0$ such that for any $x,y\in M$, if $$d(\phi_t(x),\phi_t(y))<\delta$$
for any $t\in\mathbb{R}$, then $y\in Orb(x, \phi)$.  We also call $\delta$ the separating constant of $\phi$.
\end{Definition}

Separating property is a weak form of expansivity, introduced by Gura \cite{G}. It is known that if a flow is expansive, or Komuro expansive or kinematic expansive, then it is separating. If $\phi_t$ is separating and the set of fixed points is open in $M$, then $\phi_t$ is BH-expansive which defined in \cite{KH}(see \cite{Ar}).

\begin{Definition}
Let $M$ be a compact metric space with metric $d$, $\phi:\mathbb{R}\times M\to M$ be a continuous flow. The centralizer of $\phi$ is the set of all the continuous flows commute with $\phi$, denoted by
$$\mathcal{Z}(\phi)=\{\psi|\psi:\mathbb{R}\times M\to M \ \hbox{is a continuous flow and } \ \psi_s\circ\phi_t=\phi_t\circ\psi_s \ \hbox{for any} \ t,s\in\mathbb{R}\}.$$
\end{Definition}

We say the centralizer of flow $\phi$ is {\it quasi trivial} if for any $\psi\in\mathcal{Z}(\phi)$, there exists a continuous function $A:M\to\mathbb{R}$ such that $A$ is invariant along to the orbit of $\phi$ and
$\psi_t(x)=\phi_{A(x)t}(x)$ for every $(t,x)\in\mathbb{R}\times M$. The following theorem is one of the main result of this paper which shows that every separating fixed point free flow has quasi trivial centralizer.

\begin{Theorem A}\label{thmA}
Let $M$ be a compact metric space with metric $d$, $\phi:\mathbb{R}\times M\to M$ be a continuous flow without fixed points. If $\phi_t$ is separating, then $\phi_t$ has quasi trivial centralizer.
\end{Theorem A}

We also consider the centralizer of separating fixed point free $\mathbb{R}^d$-actions. As usual we say that a continuous map $\Phi:\mathbb{R}^d\times M\to M$ is a continuous {\it $\mathbb{R}^d$-action} on a compact metric space $M$ if $\Phi(0, x)=x$ for all $x\in M$ and
$\Phi_v\triangleq\Phi(v,\cdot):M\to M$ is a homeomorphism on $M$ and $\Phi_{v+u}=\Phi_v\circ\Phi_u$ for every $v,u\in\mathbb{R}^d$. Given $x\in M$, the set $Orb(x, \Phi)=\{\Phi_v(x)|v\in\mathbb{R}^d\}$ is called the orbit of $x$ with respect to $\Phi$. If $M$ is a $C^r$ manifold and a continuous $\mathbb{R}^d$-action $\Phi$ is also $C^r$ ($r\geq 1$), then we say $\Phi$ is a $C^r$ $\mathbb{R}^d$-action.

By applying Theorem A, we can prove the following result of $\mathbb{R}^d$-actions.

\begin{Corollary}\label{cor1.3}
Take $d\geq1$ and let $\Phi:\mathbb{R}^d\times M\to M$ be a continuous $\mathbb{R}^d$-action on a compact metric space $M$. If there exists $v\in\mathbb{R}^d$ such that $(\Phi_{tv})_{t\in\mathbb{R}}$ is a fixed point free separating flow then the orbits of $\Phi$ are curves in $M$ and coincide with the orbits of flow $(\Phi_{tv})_{t\in\mathbb{R}}$.
\end{Corollary}

Similar to separating flows, we can define separating property for $\mathbb{R}^d$ actions.

\begin{Definition}
Let $M$ be a compact metric space with metric $d$ and $\Phi:\mathbb{R}^d\times M\to M$ be a continuous $\mathbb{R}^d$-action. We say that $\Phi$ is separating if there exists a separating constant $\delta>0$ such that for any $x,y\in M$ satisfying $d(\Phi_v(x),\Phi_v(y))<\delta$ for all $v\in\mathbb{R}^d$, $y$ belongs to the orbit of $x$ under the action of $\Phi$.
\end{Definition}

Similarly, given a $C^r$-action $\Phi:\mathbb{R}^d\times M\to M$, $r\geq0$, we define its $C^k$($0\leq k\leq r$) centralizer as the set
$$\mathcal{Z}^k(\Phi)=\{\Psi:\mathbb{R}^d\times M\to M \text{ is a } C^k \text{ action } |\Phi_v\circ\Psi_u=\Psi_u\circ\Phi_v,\forall v, u\in\mathbb{R}^d\}.$$
We say that $\Phi$ has a $C^k$ quasi-trivial centralizer if for any $\Psi\in\mathcal{Z}^k(\Phi)$ there exists a continuous map $A: M\to\mathcal{M}_{d\times d}(\mathbb{R})$ satisfying $A(x)=A(\Phi_v(x))$ for every $v\in\mathbb{R}^d$, $x\in M$ and so that $$\Psi_v(x)=\Phi_{A(x)v}(x)$$ for every $(v,x)\in\mathbb{R}^d\times M$, where $\mathcal{M}_{d\times d}(\mathbb{R})$ is the space of all $d\times d$ real matrices.

\begin{Definition}
Let $M$ be a compact Riemannian manifold. We say that a $C^1$ $\mathbb{R}^d$-action $\Phi:\mathbb{R}^d\times M\to M$ is homogeneous (or locally free) if all orbits by $\Phi$ are $d$-dimensional immersion submanifolds on $M$ with dimension equal to $d$.
\end{Definition}

It is known that not every manifold admits homogeneous $\mathbb{R}^d$-actions (for example, Poincar$\rm{\acute{e}}$-Bendixson theorem told us that every $C^1$-flow on $\mathbb{S}^2$ admits singularities). On the other hand, the only $n$-dimensional manifold that support homogeneous $\mathbb{R}^n$-actions is torus $\mathbb{T}^n$ (see \cite{A}), and the space of homogeneous $\mathbb{R}^d$-actions forms an open subset of all $\mathbb{R}^d$-actions.

In this paper we also prove the following theorem for separating $C^1$ $\mathbb{R}^d$-actions, which can be seen an improving of Theorem 2 of \cite{BRV}.

\begin{Theorem B}
Let $M$ be a compact Riemannian manifold without boundary and $\Phi: \mathbb{R}^d\times M\to M$ be a $C^1$ $\mathbb{R}^d$-action on $M$, $\dim M\geq d$. If $\Phi$ is separating and
homogeneous then $\mathcal{Z}^0(\Phi)$ is quasi-trivial, i.e., for every $\Psi\in\mathcal{Z}^0(\Phi)$ there exists a continuous-map $A:M\to\mathcal{M}_{d\times d}(\mathbb{R})$  satisfying $A(x)=A(\Phi_v(x))$ for every $(v,x)\in\mathbb{R}^d\times M$ and so that $\Psi_v(x)=\Phi_{A(x)v}(x)$ for every $(v,x)\in\mathbb{R}^d\times M$.
\end{Theorem B}

\section{Centralizer of separating fixed point free flow}

Before we start the proof of Theorem A, we prepare several lemmas for fixed point free flows. While some of these lemmas are well-known, we provide proofs here for the sake of completeness. For continuous flow $\phi$ on a metric space $M$, we define the corresponding constant $\varepsilon_0(\phi)$ as follows
$$\varepsilon_0(\phi)=\inf(\{\tau>0|\exists \ x\in M \ \text{ s. t. } \ \phi_{\tau}(x)=x, \text{ and }\phi_t(x)\neq x \ \forall \ t\in(0,\tau)\}\cup\{1\}).$$
By the definition of $\varepsilon_0(\phi)$, it is easy to see that for any $t\in(-\varepsilon_0(\phi), \varepsilon_0(\phi))$, if one can find a point $x\in M$ such that $\phi_t(x)=x$, then $t=0$.

Throughout the rest of the paper, we assume that the flow $\phi_t$ has no fixed points. Denote by $B_\delta(x)$ the ball centered at $x$ in $M$ of radius $\delta$, and $\bar{B}_\delta(x)$ the closure of $B_\delta(x)$.

\begin{Lemma}\label{lem1}
Let $\phi:\mathbb{R}\times M\to M$ be a continuous flow without fixed points, then $\varepsilon_0(\phi)>0$.
\end{Lemma}
\begin{proof}
Suppose the contrary that we have $\varepsilon_0(\phi)=0$. By the definition of the infimum, for any $\varepsilon>0$, there exists $t'\in\mathbb{R}^+$ and $x'\in M$ with $0<t'<\varepsilon$ and $\phi_{t'}(x')=x'$. Taking $\varepsilon=1,\frac{1}{2},\cdots,\frac{1}{n},\cdots$, we can get two sequences $\{t_n|n\geq1\}\subset\mathbb{R}^+$ and $\{x_n|n\geq1\}\subset M$ such that $0<t_n<\frac{1}{n}$ and $\phi_{t_n}(x_n)=x_n$. Since $M$ is compact, one can take a subsequence $\{x_{n_k}|k\geq1\}\subset\{x_n|n\geq1\}$ satisfied $x_{n_k}\to x_0\in M$ corresponding to $t_{n_k}\to0, \  \phi_{t_{n_k}}(x_{n_k})=x_{n_k}$. Take any $t\in\mathbb{R}$ fixed, for any $t_{n_k}$, there exists a unique $l_{n_k}\in\mathbb{Z}$ such that
$$l_{n_k}t_{n_k}<t\leq(l_{n_k}+1)t_{n_k},$$
then $0<t-l_{n_k}t_{n_k}\leq t_{n_k}$ for all $n_k$ and $\lim\limits_{n_k\to\infty}(t-l_{n_k}t_{n_k})=0.$ Therefore $$\phi_t(x_{n_k})\in\phi_{(l_{n_k}t_{n_k},(l_{n_k}+1)t_{n_k}]}(x_{n_k}),$$
and $\phi_{l_{n_k}t_{n_k}}(x_{n_k})=\phi_{t_{n_k}}(x_{n_k})=\phi_{(l_{n_k}+1)t_{n_k}}(x_{n_k})$, by the continuity of $\phi_t$, we must have $$\phi_t(x_0)=\lim_{n_k\to\infty}\phi_{l_{n_k}t_{n_k}}(x_{n_k})=\lim_{n_k\to\infty}\phi_{t_{n_k}}(x_{n_k})=x_0.$$
According to the arbitrariness of $t$, we have $\phi_t(x_0)=x_0$ for any $t\in\mathbb{R}$. It means that $x_0$ is a fixed point of $\phi_t$, that is a contradiction.
\end{proof}

\begin{Lemma}\label{lem2}
Let $\phi:\mathbb{R}\times M\to M$ be a continuous flow without fixed points. For any $T\in(0,\varepsilon_0(\phi))$, we have a constant $\eta>0$ with $d(\phi_{T}(x),x)\geq\eta$ for any $x\in M$.
\end{Lemma}

\begin{proof}
Suppose the result is false, there exists $t\in(0,\varepsilon_0(\phi))$ such that for any $\eta>0$, we have $x\in M$ with $d(\phi_t(x),x)<\eta$. By taking $\eta=\frac{1}{n}$ in turn, we obtain a sequences $\{x_n\}\subset M$ satisfied $d(\phi_{t}(x_n),x_n)<\frac{1}{n}$. Since $M$ is compact metric space, we can take subsequence $\{x_{n_k}\}$ of $\{x_n\}$ such that $x_{n_k}\to x\in M$. Then we get $d(\phi_t(x),x)=0$, it means that $\phi_t(x)=x$, that contradicts with the choice of $\varepsilon_0(\phi)$.
\end{proof}

\begin{Lemma}\label{cla1}
Let $T_0\in(0, \varepsilon_0(\phi))$ be given, $\eta$ is the constant corresponding to $T_0$ in Lemma \ref{lem2}. There exist $\mu_1\in(0,T_0),\delta>0$ such that for every $x\in M, p\in\bar{B}_{\delta}(x), t\in[-\mu_1, \mu_1], \tau\in[0,T_0]$, we have
\begin{itemize}
  \item[$(1)$] $d(\phi_t(p),x)<\frac{\eta}{4}$;
  \item[$(2)$] $d(\phi_{\tau}(p),\phi_{\tau}(x))\leq\frac{\eta\mu_1}{12T_0}(<\frac{\eta}{12})$;
  \item[$(3)$] $d(\phi_{\pm\mu_1/3}(x), x)\geq 2\delta$.
\end{itemize}
\end{Lemma}
\begin{proof}
We first prove that there exist $\mu_1\in(0,\frac{T_0}{3}), \delta_1>0$ small enough such that for every $x\in M,p\in\bar{B}_{\delta_1}(x), t\in[-\mu_1, \mu_1],$
$$d(\phi_t(p),x)<\frac{\eta}{4}.$$
Assume the contrary holds, we can find sequences $\{x_n\}$ in $M$, $\{p_n\}$ in $\bar{B}_{\frac{1}{n}}(x_n)$ and $\{t_n\}$ in $[-\frac{1}{n}, \frac{1}{n}]$ such that
$$d(\phi_{t_n}(p_n),x_n)\geq \frac{\eta}{4}.$$
By compactness of $M$, we can take a subsequence $\{x_{n_k}\}$ of $\{x_n\}$ such that $x_{n_k}\to x\in M$. We can see that $p_{n_k}\to x$ and $t_{n_k}\to 0$. Then we have $\phi_{t_{n_k}}(p_{n_k})\to x$ and $d(\phi_{t_{n_k}}(p_{n_k}), x_{n_k})\to 0$. This contradicts with $d(\phi_{t_{n_k}}(p_{n_k}),x_{n_k})\geq\frac{\eta}{4}$ for all $k$.

Let constant $\mu_1\in(0, \frac{T_0}{3})$ be chosen as above. Now we show that there is $0<\delta_2<\delta_1$ such that for every $x\in M, p\in\bar{B}_{\delta_2}(x), \tau\in[0,T_0]$, we have
$$d(\phi_{\tau}(p),\phi_{\tau}(x))\leq\frac{\eta\mu_1}{12T_0}.$$
Assume the contrary, then there exist sequences $x_n\in M, p_n\in\bar{B}_{\frac{1}{n}}(x_n),\tau_n\in[0,T_0]$ such that $$d(\phi_{\tau_n}(p_n),\phi_{\tau_n}(x_n))\geq\frac{\eta\mu_1}{12T_0}.$$ We can take subsequences such that $x_{n_k}\to x\in M, \tau_{n_k}\to\tau\in[0,T_0]$. Let $k\to\infty$ we have
$$0=d(\phi_{\tau}(x),\phi_{\tau}(x))\geq\frac{\eta\mu_1}{12T_0}.$$
This is a contradiction. By applying Lemma \ref{lem2}, we can find $0<\delta<\delta_2$ small enough such that $$d(\phi_{\pm\mu_1/3}(x), x)\geq 2\delta$$ for any $x\in M$. One can check that the constants $\mu_1$ and $\delta$ satisfy the request of the lemma.
\end{proof}

The following lemma says that we can construct ``cross sections'' for continuous flows on a compact metric space. The idea follows \cite{Wh}. One can see similar idea also in \cite{Ar2}.

\begin{Lemma}\label{lem3}
Let $\phi:\mathbb{R}\times M\to M$ be a continuous flow without fixed points, $\varepsilon_0(\phi)$ be the constant for $\phi_t$ in Lemma \ref{lem2}. There exists a family $$\{S_x\subset M \text{ is a close set} | x\in M\}$$ with constants  $\mu\in(0,\frac{\varepsilon_0(\phi)}{3}),\delta>0$ such that for any $x\in M$, there is a continuous function $\tau:\bar{B}_{\delta}(x)\to[-\mu,\mu]$ with $\tau(0)=0$ and $\phi_{\tau(p)}(p)\in S_x$ for every $p\in\bar{B}_{\delta}(x)$. Moreover, for any $x\in M$ and $p\in B_{\delta}(x)$, there exist $-\mu\leq l_1<0<l_2\leq\mu$ such that $\phi_{l_1}(p)\in\partial B_\delta(x), \phi_{l_2}(p)\in\partial B_\delta(x), \phi_{(l_1, l_2)}(p)\subset B_\delta(x)$ and
for any $q\in\varphi_{[l_1, l_2]}(p)$,  we have $\phi_{\tau(q)}(q)=\phi_{\tau(p)}(p)$.
\end{Lemma}

\begin{proof}
Fix a constant $T_0\in(0,\varepsilon_0(\phi))$. For any $x,p\in M$, we can define a function $I(p,x)$ as following: $$I(p,x)\triangleq\int^{T_0}_0d(\phi_{s}(p),x){\rm d}s.$$
Denote that $S_x=\{p\in M|I(p,x)-I(x,x)=0\}$, according to the continuity of $I(p,x)$, we know that $S_x$ is a closed subset in $M$. Fixed a point $x\in M$, we can define a function $G(t,p)$ as following:
$$G(t,p)\triangleq I(\phi_t(p),x)-I(x,x).$$
Now we prove that $G(t,p)$ is differentiable with respect to $t$. Let $t\in\mathbb{R}, p\in M, \Delta t\neq 0$ be given, we have
\begin{eqnarray*}
  \frac{G(t+\Delta t,p)-G(t,p)}{\Delta t} &=& \frac{\int^{T_0}_0(d(\phi_{t+\Delta t+s}(p),x)-d(\phi_{t+s}(p),x)){\rm d}s}{\Delta t} \\
   &=& \frac{\int^{t+\Delta t+T_0}_{t+\Delta t} d(\phi_{s}(p),x){\rm d}s-\int^{t+T_0}_{t}d(\phi_{s}(p),x){\rm d}s}{\Delta t} \\
   &=& \frac{\int^{t+T_0+\Delta t}_{t+T_0}d(\phi_{s}(p),x){\rm d}s-\int^{t+\Delta t}_{t}d(\phi_{s}(p),x){\rm d}s}{\Delta t}
\end{eqnarray*}
Then by the integral mean value theorem and continuity of $\phi$, we know that
$$\frac{\partial G}{\partial t}(t,p)=\lim_{\Delta t\to0}\frac{G(t+\Delta t,p)-G(t,p)}{\Delta t}=d(\phi_{t+T_0}(p),x)-d(\phi_{t}(p),x).$$
Let $\eta$ be given as in Lemma \ref{lem2} and $\mu_1,\delta$ be the constants given in Lemma \ref{cla1} associated to $T_0$. By the choice of $\eta, \mu_1, \delta$, we know that $$d(\phi_{t+T_0}(p),x)-d(\phi_{t}(p),x)\geq d(\phi_{t+T_0}(p), \phi_t(p))-2d(\phi_{t}(p),x)>\eta- 2\cdot\frac{\eta}{4}=\frac{\eta}{2}$$ for any $x\in M,p\in\bar{B}_{\delta}(x)$ and $t\in[-\mu_1,\mu_1]$. That means $\frac{\partial G}{\partial t}(t,p)>\frac{\eta}{2}>0$ for any $x\in M,p\in\bar{B}_{\delta}(x),t\in[-\mu_1,\mu_1]$,
in other word, for any $x\in M,p\in\bar{B}_{\delta}(x)$, $G(t,p)$ is strictly monotone increasing with respect to $t\in[-\mu_1,\mu_1]$.

Now let $\mu=\mu_1/3$. Denote by $C(\bar{B}_{\delta}(x),\mathbb{R})$ the set of continuous functions defined on $\bar{B}_\delta(x)$ and $C(\bar{B}_{\delta}(x),[-\mu,\mu])$ the set of continuous functions $\tau:\bar{B}_\delta(x)\to[-\mu,\mu]$.   We consider the map $$F:\mathcal{X}\to C(\bar{B}_{\delta}(x),\mathbb{R})$$
$$(F\tau)(p)=\tau(p)-(\frac{\partial G}{\partial t}(0,x))^{-1}G(\tau(p),p),$$
where $\mathcal{X}=\{\tau\in C(\bar{B}_{\delta}(x),[-\mu,\mu])|\tau(x)=0\}$. Note that $\mathcal{X}\subset C(\bar{B}_{\delta}(x),\mathbb{R})$ is a closed subset in $C(\bar{B}_{\delta}(x),\mathbb{R})$, therefore $\mathcal{X}$ is complete metric space with respect to the metric $\rho$ in $C(\bar{B}_{\delta}(x),\mathbb{R})$, where $\rho$ is defined by $$\rho(\tau_1,\tau_2)=\max_{p\in\bar{B}_{\delta}(x)}|\tau_1(p)-\tau_2(p)|.$$
For any $\tau_1,\tau_2\in\mathcal{X}$, according to the mean value theorem of differentiation, we have
\begin{eqnarray*}
  \rho(F\tau_1,F\tau_2) &=& \max_{p\in\bar{B}_{\delta}(x)}|\tau_1(p)-(\frac{\partial G}{\partial t}(0,x))^{-1}G(\tau_1(p),p)-\tau_2(p)+(\frac{\partial G}{\partial t}(0,x))^{-1}G(\tau_2(p),p)| \\
   &=& \max_{p\in\bar{B}_{\delta}(x)}|(\tau_1(p)-\tau_2(p))(1-(\frac{\partial G}{\partial t}(0,x))^{-1}\frac{\partial G}{\partial t}(\zeta,p))|,
\end{eqnarray*}
where  $\zeta\triangleq\zeta(p)=\theta(p)\tau_1(p)+(1-\theta(p))\tau_2(p)\in\mathcal{X}, \ \theta(p)\in[0,1], \ p\in\bar{B}_{\delta}(x)$. According to conclusions of Lemma \ref{cla1} and Lemma \ref{lem2}, we have
$$|d(\phi_{T_0}(x),x)-d(\phi_{\zeta+T_0}(p),x)|\leq d(\phi_{T_0}(x), \phi_{T_0}(p))+d(\phi_{T_0}(p), \phi_{\zeta+T_0}(p))\leq\frac{\eta\mu_1}{12T_0}+\frac{\eta}{4}<\frac{\eta}{3}$$
and then
\begin{eqnarray*}
  \left|1-(\frac{\partial G}{\partial t}(0,x))^{-1}\frac{\partial G}{\partial t}(\zeta,p)\right| &=& \left|1-\frac{d(\phi_{\zeta+T_0}(p),x)-d(\phi_t(p),x)}{d(\phi_{T_0}(x),x)}\right| \\
   &=& \left|\frac{d(\phi_{T_0}(x),x)-d(\phi_{\zeta+T_0}(p),x)+d(\phi_t(p),x)}{d(\phi_{T_0}(x),x)}\right| \\
   &\leq& \frac{|d(\phi_{T_0}(x),x)-d(\phi_{\zeta+T_0}(p),x)|}{d(\phi_{T_0}(x),x)}+\frac{d(\phi_t(p),x)}{d(\phi_{T_0}(x),x)} \\
   &\leq& \frac{\frac{\eta}{3}}{\eta}+\frac{\frac{\eta}{4}}{\eta}=\frac{7}{12}.
\end{eqnarray*}
Now we know that $$\rho(F\tau_1,F\tau_2)\leq\frac{7}{12}\rho(\tau_1,\tau_2),$$
hence $F$ is a contraction map. In the following we verify that $F$ is a mapping from $\mathcal{X}$ to itself. For any function $\tau\in\mathcal{X}$, we have
$$F\tau(x)=\tau(x)-(\frac{\partial G}{\partial t}(0,x))^{-1}G(\tau(x),x)=0,$$
and
\begin{eqnarray*}
  \rho(F\tau,0) &\leq& \rho(F\tau,F0)+\rho(F0,0) \\
   &\leq& \frac{7}{12}\rho(\tau,0)+\max_{p\in\bar{B}_{\delta}(x)}|(\frac{\partial G}{\partial t}(0,x))^{-1}G(0,p)| \\
   &=& \frac{7}{12}\rho(\tau,0)+\max_{p\in\bar{B}_{\delta}(x)}\frac{|\int_0^{T_0}(d(\phi_{s}(p),x)-d(\phi_{s}(x),x)){\rm d}s|}{d(\phi_{T_0}(x),x)} \\
   &\leq& \frac{7}{12}\rho(\tau,0)+\max_{p\in\bar{B}_{\delta}(x)}\frac{|\int_0^{T_0}d(\phi_{s}(p),\phi_{s}(x)){\rm d}s|}{\eta} \\
   &\leq& \frac{7}{12}\rho(\tau,0)+\frac{\mu_1\eta T_0}{12\eta T_0}\leq\frac{7}{12}\mu+\frac{1}{4}\mu<\mu.
\end{eqnarray*}
So $F(\mathcal{X})\subset\mathcal{X}$, and $F:\mathcal{X}\to\mathcal{X}$ is a contraction mapping. By the contraction mapping theorem we know that there exists an unique continuous function $\tau\in\mathcal{X}$ such that $G(\tau(p),p)=0$ for any $x\in M, p\in\bar{B}_{\delta}(x)$. It is equivalent to $\phi_{\tau(p)}(p)\in S_x$ for any $x\in M, p\in\bar{B}_{\delta}(x)$.

No we prove that given any $x\in M$ and $p\in B_{\delta}(x)$, there exist $-\mu\leq l_1<0<l_2\leq\mu$ such that $\phi_{l_1}(p)\in\bar{B}_\delta(x), \phi_{l_2}(p)\in\bar{B}_\delta(x), \phi_{(l_1, l_2)}(p)\subset B_\delta(x)$ and
for any $q\in\phi_{[l_1, l_2]}(p)$,  we have $\phi_{\tau(q)}(q)=\phi_{\tau(p)}(p)$.

Let $x\in M$ and $p\in B_{\delta}(x)$ be given, denote by $$l_1=\sup\{t<0|\phi_t(p)\in\partial B_{\delta}(x)\},$$ $$l_2=\inf\{t>0|\phi_t(p)\in\partial B_{\delta}(x)\}.$$ By the choice of $\delta$, we have $d(\phi_{\pm\mu}(p), x)>2\delta$ and then $\phi_{\pm\mu}(p)\notin\bar{B}_\delta(x)$, thus $l_1,l_2$ exist and $-\mu<l_1<0<l_2<\mu$. It is easy to check that $\phi_{[l_1,l_2]}(p)\subset\bar{B}_{\delta}(x)$. For any $q\in\phi_{[l_1,l_2]}(p)$, without loss of generality, let $q=\phi_{t_1}(p)$, then we have $p=\phi_{-t_1}(q)$, where $t_1\in[-\mu,\mu]$. Since $p,q\in B_{\delta}(x)$, there exist $\tau(p), \tau(q)\in[-\mu,\mu]$ such that $\phi_{\tau(p)}(p)\in S_x, \phi_{\tau(q)}(q)\in S_x$.  Note that
$$\phi_{\tau(p)}(p)=\phi_{\tau(p)}(\phi_{-t_1}(q))=\phi_{\tau(p)-t_1}(q)$$
Hence both $\tau(p)-t_1$ and $\tau(q)$ are real numbers $s$ in $[-\mu_1,\mu_1]$ with properties $\phi_s(q)\in S_x$. Note that $\phi_s(q)\in S_x$ is equivalent to $G(s, q)=0$. We have already known that $G(s, q)$ is strictly increasing with respect to $s\in[-\mu_1,\mu_1]$, hence there is at most one real number $s\in[-\mu_1, \mu_1]$ such that $\phi_s(q)\in S_x$. Thus we have $\tau(p)-t_1=\tau(q)$, and then $\phi_{\tau(q)}(q)=\phi_{\tau(p)}(p)$.
\end{proof}

\begin{Lemma}\label{lem4}
Let $M$ be a metric space and $\phi_t$ be a fixed point-free flow on $M$. There exist constants $\mu\in(0, \frac{\varepsilon_0(\phi)}{3}), \delta>0$ such that for any $x\in M$, and any continuous curve $\alpha:[0,1]\to B_{\delta}(x)$ with $\alpha(0)=x$ and $\alpha(t)\in Orb(x, \phi)$ for all $t\in[0,1]$, we have $\alpha(t)\in\phi_{[-\mu,\mu]}(x)$ for all $t\in[0,1]$.
\end{Lemma}

\begin{proof}
Let $\mu,\delta$ and the family $\{S_x\}$ be given as above Lemma \ref{lem3}. Let $x\in M$ and $\alpha:[0,1]\to B_{\delta}(x)$ be a continuous curve in $M$ with $\alpha(0)=x$ and $\alpha(t)\in Orb(x, \phi)$ for all $t\in[0,1]$. Let $\tau:\bar{B}_\delta(x)\to[-\mu,\mu]$ be the continuous function with $\phi_{\tau(p)}(p)\in S_x$ for any $p\in\bar{B}_\delta(x)$. Define a continuous map \begin{eqnarray*}
  P_x:B_{\delta}(x) &\to& S_x \\
  p &\mapsto& \phi_{\tau(p)}(p).
\end{eqnarray*}
We can see that $P_x\circ\alpha$ is a continuous curve on $S_x$ and $P_x\circ\alpha(0)=x$. Note that $$\{t\in\mathbb{R}|\phi_t(x)\in B_{\delta}(x)\}$$
is an open set in $\mathbb{R}$, therefore it is a union of countable open intervals $\{I_j\}$. For each interval $I_j=[l_1, l_2]$, we know that $\phi_{l_1}(x), \phi_{l_2}(x)\in \partial B_\delta(x)$ and $\phi_{(l_1, l_2)}(x)\subset B_\delta(x)$, by the conclusion of Lemma \ref{lem3}, we know each
$P_x(\phi_{I_j}(x))$ is a single point in $S_x$. Now we know that the cardinal number of $P_x\circ\alpha([0,1])$ is at
mostly countable. If there is $t\in[0,1]$ such that $P_x\circ\alpha(t)\neq x$, then $d(P_x\circ\alpha(t), x)$ formulate an interval in $\mathbb{R}$, contradicts with that $P_x\circ\alpha([0,1])$ is at mostly countable. Hence $P_x\circ\alpha([0,1])=\{x\}$. Then we know that $\alpha(t)=\phi_{-\tau(\alpha(t))}(x)$ for any $t\in[0,1]$. We know that $-\tau(\alpha(t))\in[-\mu,\mu]$, hence $\alpha(t)\in\phi_{[-\mu,\mu]}(x)$ for any $t\in[0,1]$. This proves the lemma.
\end{proof}

\begin{Lemma}\label{lem5}
Let $M$ be a compact metric space with metric $d$, $\phi:\mathbb{R}\times M\to M$ be a continuous flow without fixed points. Suppose $\phi$ is separating. Given any $\psi\in\mathcal{Z}(\phi)$, there exist a constant $a>0$ and a function $z:[-a,a]\times M\to [-\mu,\mu]$ such that $\psi_s(x)=\phi_{z(s, x)}(x)$ for every $(s, x)\in
[-a,a]\times M$. Moreover,
\begin{enumerate}
\item $z$ is continuous;
\item $z(t+s, x)=z(t,x)+z(s, \psi_t(x))$ for any $x\in M$ and $t,s\in[-a,a]$ with $t+s\in[-a,a]$;
\item $z(s, \phi_t(x))=z(s, x)$ for any $x\in M$ and $s\in[-a,a]$ and $t\in\mathbb{R}$;
\item $z(s, x)=A(x)s$ for any $x\in M$ and $s\in[-a,a]$, where $A(x)=a^{-1}z(a,x)$.
\end{enumerate}
\end{Lemma}

\begin{proof}
Let $\mu\in(0, \frac{\varepsilon_0(\phi)}{3}), \delta>0$ be given as in Lemma \ref{lem4}. Since $\phi_t$ is separating, without loss of generality, we can assume that $\delta>$ is also chosen small enough such that for any $x,y\in M$, if $d(\phi_t(y), \phi_t(x))<\delta$ for all $t\in\mathbb{R}$ hold, then $y\in Orb(x,\phi)$.

Let $\psi\in\mathcal{Z}(\phi)$ be given. Since $\psi$ is a continuous flow on $M$, according to compactness of $M$, there exists a constant $a>0$ such that
$$d(\psi_s(x),x)<\delta$$
for all $(s,x)\in[-a,a]\times M$. By the assumption that $\psi\in\mathcal{Z}(\phi)$, one has $$d(\phi_t(x),\phi_t(\psi_s(x)))=d(\phi_t(x),\psi_s(\phi_t(x)))<\delta$$
for all $s\in[-a,a],t\in\mathbb{R},x\in M$. Since $\phi$ is separating, we have $\psi_s(x)\in Orb(\phi_t(x),\phi)$. Note that $\psi_s(x)\in B_{\delta}(x)$ for all $s\in[-a,a]$ and $\psi_0(x)=x$. By Lemma \ref{lem4}, there exists $\eta=z(s,x)\in[-\mu,\mu]$ such that $\psi_s(x)=\phi_{\eta}(x)$. Because of $\mu\in(0, \frac{\varepsilon_0(\phi)}{3})$, we know that for any $s_1,s_2\in[-\mu,\mu]$, when $s_1\neq s_2$, then $\phi_{s_1}(x)\neq\phi_{s_2}(x)$ for any $x\in M$. Hence $\eta=z(s,x)$ is uniquely determined for every $(s,x)\in[-a,a]\times M$. We define a function $\eta=z(s,x)$ on $[-a,a]\times M$.

We now proceed to demonstrate the continuity of $z(s,x)$. Let $\{(s_n,x_n)\in[-a,a]\times M\}$ be a sequence with $(s_n, x_n)\to (s_0, x_0)\in[-a,a]\times M$ as $n\to\infty$. By definition, we have
$$\phi_{z(s_n,x_n)}(x_n)=\psi_{s_n}(x_n),$$
$$\phi_{z(s_0,x_0)}(x_0)=\psi_{s_0}(x_0).$$
Since $\psi$ and $\phi$ are continuous flows, we have $$\lim\limits_{n\to\infty}\phi_{z(s_n,x_n)}(x_n)=\lim\limits_{n\to\infty}\psi_{s_n}(x_n)=\psi_{s_0}(x_0)=\phi_{z(s_0,x_0)}(x_0).$$
Suppose that $\lim\limits_{n\to\infty}z(s_n,x_n)=\eta_0\neq z(s_0,x_0)$. Then, we have $\phi_{\eta_0}(x_0)=\phi_{z(s_0,x_0)}(x_0)$, which implies that $x_0$ is a period point with period $z(s_0,x_0)-\eta_0\in[-2\mu,2\mu]\subset(-\varepsilon_0(\phi), \varepsilon_0(\phi))$.
This contradicts with the choice of $\varepsilon_0(\phi)$. Therefore, we conclude that $z(s,x)$ is continuous.

Note that
\begin{eqnarray*}
  \phi_{z(t+s,x)}(x) &=& \psi_{t+s}(x)=\psi_s(\psi_t(x))=\phi_{z(s, \psi_t(x))}(\psi_t(x)) \\
   &=& \phi_{z(s, \psi_t(x))}(\phi_{z(t,x)}(x))=\phi_{z(t,x)+z(s,\psi_t(x))}(x),
\end{eqnarray*}
and then  $\phi_{z(t+s,x)-(z(t,x)+z(s,\psi_t(x)))}(x)=0$. It is easy to see that $z(t+s, x)-(z(t,x)+z(s,\psi_t(x)))\in[-3\mu, 3\mu]\subset(-\varepsilon_0(\phi), \varepsilon_0(\phi))$, then by the choice of $\varepsilon_0(\phi)$ we can see that $z(t+s,x)=z(t,x)+z(s,\psi_t(x))$ for any $x\in M$ and $t,s\in[-a,a]$ with $t+s\in[-a,a]$. This proves item 2 of the lemma.

Fix $s\in[-a,a]$ and $t\in[-\mu,\mu]$ and $x\in M$. Note that
\begin{eqnarray*}
  \phi_{t+z(s, \phi_t(x))}(x)&=&\phi_{z(s, \phi_t(x))}(\phi_t(x))=\psi_s(\phi_t(x)) \\
   &=& \phi_t(\psi_s(x))=\phi_t(\phi_{z(s,x)}(x))=\phi_{t+z(s, x)}(x),
\end{eqnarray*}
and then $\phi_{z(s, \phi_t(x))-z(s, x)}(x)=x$.
It is easy to check that $z(s, \phi_t(x))-z(s, x)\in[-2\mu,2\mu]\subset(-\varepsilon_0(\phi), \varepsilon_0(\phi))$ we know that $z(s, \phi_t(x))-z(s, x)=0$ and then $z(s, \phi_t(x))=z(s,x)$.

Let $s\in[-a,a]$ and $t\in\mathbb{R}$ and $x\in M$ be given. We can find $n\in\mathbb{N}$ big enough such that $|n^{-1}t|\leq\mu$, then we can see that
$$z(s, \phi_{\frac{1}{n}t}(x))=z(s,x).$$
For $\phi_{\frac{1}{n}t}(x)$, also have
$$z(s, \phi_{\frac{2}{n}t}(x))=z(s,\phi_{\frac{1}{n}t}(\phi_{\frac{1}{n}t}(x)))=z(s,\phi_{\frac{1}{n}t}(x))$$
Inductively, we have
$$z(s,\phi_t(x))=z(s,\phi_{\frac{n-1}{n}t}(x))=\cdots=z(s,\phi_{\frac{1}{n}t}(x))=z(s,x)$$
This proves item 3 of the lemma.

From item 2 and 3 we can see that for any $s, t\in[-a,a]$ and any $x\in M$, we have
$$z(s+t, x)=z(t, x)+z(s, \psi_t(x))=z(t, x)+z(s, \phi_{z(t, x)}(x))=z(t, x)+z(s, x).$$
Fix $x\in M$. For any $n\in\mathbb{Z}^+$, we have $nz(n^{-1}a, x)=z(a, x)$, and then we have $$z(n^{-1}a, x)=n^{-1}z(a, x)=aA(x).$$ And then we have $$z(\frac{m}{n}a, x)=\frac{m}{n}a A(x)$$ for any rational number $\frac{m}{n}\in[0, 1]$. Note that $z(-t, x)=-z(t, x)$, we can see that $$z(\frac{m}{n}a, x)=\frac{m}{n}a A(x)$$ for any rational number $\frac{m}{n}\in[-1,1]$, by the continuity of $z(s, x)$ we can see that $$z(s, x)=A(x)s$$ for any $s\in[-a,a]$. This prove item 4.
\end{proof}

We now proceed to prove Theorem A.

\bigskip

{\noindent\it  Proof of Theorem A.}
Let $\psi\in\mathcal{Z}(\phi)$ be given, that is, $\psi_t$ is a continuous flow commute with $\phi_t$. Then we can take $a>0$ and a function $z(s, x)=A(x)s$ as in Lemma \ref{lem5}.
By the continuity of $z(s, x)$ we can see that $A(x)$ is continuous on $M$. Let $x\in M$ and $t\in\mathbb{R}$ be given, we have
$$A(\phi_t(x))=a^{-1}z(a, \phi_t(x))=a^{-1}z(a, x)=A(x).$$
Hence $A(x)$ is constant along orbit of $\phi_t$.

By the result of Lemma \ref{lem5} that $\psi_t(x)=\phi_{z(t,x)}(x)$ for any $(t,x)\in[-a,a]\times M$ we can easily establish that $\psi_t(x)=\phi_{A(x)t}(x)$ holds for any $x\in M$ and $t\in[-a,a]$.
We now proceed to extend $t\in[-a,a]$ to the entire real number line. Fixed any $x\in M$ and $t\in\mathbb{R}$, we can take $n\in\mathbb{N}$ big enough such that $|n^{-1}t|\leq a$. Denote by $\tau=n^{-1}t$, then we have
$$\psi_{2\tau}(x)=\psi_\tau(\psi_\tau(x))=\phi_{A(\psi_{\tau}(x))\tau}(\phi_{A(x)\tau}(x))=\phi_{A(\phi_{A(x)\tau}(x))\tau}(\phi_{A(x)\tau}(x))=\phi_{2A(x)\tau}(x),$$
and then by induction we have $$\psi_t(x)=\psi_{n\tau}(x)=\phi_{nA(x)\tau}(x)=\phi_{A(x)t}(x).$$
Therefore, we have established that $\psi_t(x)=\phi_{A(x)t}(x)$ for any $x\in M$ and $t\in\mathbb{R}$. This concludes the proof of Theorem A.
\hfill\qedsymbol

\bigskip

Corollary \ref{cor1.3} follows directly from Theorem A.
\bigskip

\noindent{\it Proof of Corollary \ref{cor1.3}.} Suppose that there exists $v\in\mathbb{R}^d$ such that $(\Phi_{tv})_{t\in\mathbb{R}}$ is a fixed point free separating flow. Let $\phi_t(x)=\Phi(tv, x)$ for any $t\in\mathbb{R}$ and $x\in M$. For any $u\in\mathbb{R}^d$, let $\psi_t=\Phi_{tu}$.  Note that $$\phi_t\circ\psi_s=\Phi_{tu}\circ\Phi_{sv}=\Phi_{tu+sv}=\Phi_{sv+tu}=\Phi_{sv}\circ\Phi_{tu}=\psi_s\circ\phi_t$$ holds for any $t,s\in\mathbb{R}$.
It follows that $\psi_t$ commutes with $\phi_t$. As a consequence of Theorem A, there exists a continuous function $A:M\to\mathbb{R}$ which is invariant along the orbits of $\phi_t$ and satisfies $\psi_t(x)=\Phi(tu,x)=\phi_{A(x)t}(x)$ for all $x\in M$ and $t\in\mathbb{R}$. Thus, we have $\Phi(tu, x)\in\{\phi_t(x)| t\in\mathbb{R}\}$ for any $t\in\mathbb{R}$. Therefore, the orbit of $\Phi$ along $x$ coincides with the orbit of $\phi_t$ along $x$ in $M$. This completes the proof of the corollary.
\hfill\qedsymbol

\section{Centralizer for $C^1$ separating homogenous $\mathbb{R}^d$-actions}

This section focuses on separating homogeneous $\mathbb{R}^d$-actions. Let $\Phi_v$ denote a homogeneous $\mathbb{R}^d$-action. Similar to Lemma \ref{lem1}, we establish the following lemma for $\Phi_v$.

\begin{Lemma}\label{lem3.1}
Let $\Phi$ be a homogenous $\mathbb{R}^d$-action, then
$$\varepsilon_0(\Phi)=\inf(\{\|v\||0\neq v\in\mathbb{R}^d, \text{there exists }x\in M,\text{ such that } \Phi_v(x)=x\}\cup\{1\})>0.$$
\end{Lemma}

\begin{proof}
Suppose that $\varepsilon_0(\Phi)=0$. By definition, there exist sequences $x_n\in M$ and $u_n\in\mathbb{R}^d$ with $\|u_n\|\to 0$ such that $$\Phi_{u_n}(x_n)=x_n,$$
for any $n=1, 2, \cdots$. Without loss of generality, we may assume that $\frac{u_n}{\|u_n\|}\to u\in\mathbb{R}^d$ and $x_n\to x\in M$ by choosing subsequences.  For any $t\in\mathbb{R}$, we have
$$-\|u_n\|+|t|\cdot\|\frac{u_n}{\|u_n\|}-u\|\leq\|[\frac{t}{\|u_n\|}]u_n-tu\|\leq \|u_n\|+|t|\cdot\|\frac{u_n}{\|u_n\|}-u\|,$$
hence we obtain $[\frac{t}{\|u_n\|}]u_n\to tu$ as $n\to\infty$. Consequently, we have
$$\Phi(tu, x)=\lim_{n\to\infty}\Phi([\frac{t}{\|u_n\|}]u_n, x_n)=\lim_{n\to\infty}x_n=x.$$
This contradicts the homogeneity of $\Phi$. Therefore, we conclude that $\varepsilon_0(\Phi)>0$.
\end{proof}

Now we consider a smooth  homogenous $\mathbb{R}^d$-action $\Phi:\mathbb{R}^d\times M\to M$. Let $\{e_1,...,e_d\}$ be the standard basis of $\mathbb{R}^d$, where $$e_1=(1,0,0,\cdots, 0), e_2=(0,1,0,\cdots,0), \cdots, e_d=(0,0,0,\cdots, 1).$$  For any $i=1,2,\cdots, d$, denote by $\phi_{e_i}:=(\Phi_{te_i})_{t\in\mathbb{R}}$ the canonical flow generated by the direction $e_i$, that is, $\phi_{e_i}$ is defined as
\begin{eqnarray*}
  \phi_{e_i}:\mathbb{R}\times M &\to& M \\
  (t,x) &\mapsto& \Phi(te_i,x).
\end{eqnarray*}
Denote by $X_i=X_{e_i}$ the vector field associated to the flow $\phi_{e_i}$, that is, $$X_{i}(x)=\frac{\rm d}{{\rm d}t}\phi_{e_i}(t,x)|_{t=0}$$ for any $x\in M$. By the homogeneity of $\Phi$ we know that $Orb(x, \Phi)=\{\Phi_v(x)| v\in\mathbb{R}^d\}$ is an $d$-dimensional immersed sub-manifold of $M$ for every $x\in M$, thus $\{X_{1}(x), X_{2}(x), \cdots, X_{d}(x)\}$ is a linear basis of $T_xOrb(x,\Phi)\subset T_xM$.

Denote by $V_x=T_xOrb(x, \Phi)$ and $N_x$ the orthogonal complement of $V_x$ in $T_xM$. For any $\xi\in T_x M$, we have a unique decomposition
$$\xi=\zeta+a_1X_{1}(x)+a_2X_{2}(x)+\cdots +a_dX_{d}(x),$$
where $\zeta\in N_x$ and $(a_1, a_2, \cdots, a_d)\in\mathbb{R}^d$.

By the compactness of $M$, there is $\rho_0>0$ such that for any $\zeta\in N_x$ with $\|\zeta\|\leq \rho_0$ and $(a_1, a_2, \cdots, a_d)\in\mathbb{R}^d$, we can set
$$F_x(\zeta+a_1X_{1}(x)+a_2X_{2}(x)+\cdots +a_dX_{d}(x))=\Phi((a_1,a_2,\cdots,a_d),\exp_x\zeta).$$

Denote by $T_xM(r)=\{u\in T_xM| \|u\|<r\}$ and $B_r(x)=\exp_x(T_xM(r))$ for $r>0$. Denote by $O_\mu=\{v\in\mathbb{R}^d|\|v\|<\mu\}$ and $\bar{O}_\mu$ be the closure of $O_\mu$ for $\mu>0$. As usual, for a linear isomorphism $A$ from a Banach space $E$ to a  Banach space $F$, denote by
$$m(A)=\inf_{\xi\in E, \|\xi\|=1}\|A\xi\|.$$
We have the following lemma similar to the flowbox theorem.

\begin{Lemma}\label{lem7}
Let $M$ be a compact Riemannian manifold without boundary, $\Phi:\mathbb{R}^d\times M\to M$ be a $C^1$ homogeneous $\mathbb{R}^d$-action, $0<d\leq\dim M$. There is $r_0>0$ such that for any $x\in M$, $F_x:T_xM(r_0)\to M$ is an embedding and $m(D_pF_x)\geq1/3$ and $\|D_pF_x\|\leq3$ for every $p\in T_xM(r_0)$.
\end{Lemma}
\begin{proof}
Since $M$ is compact, we can choose $\rho_0>0$ such that for any $x\in M$, $\exp_x: T_xM(\rho_0)\to M$ is an embedding with
$$\|D_\xi\exp_x\|<\frac{3}{2},\ \ \ \ \ \  m(D_\xi\exp_x)>\frac{2}{3}$$
for any $\xi\in T_xM(\rho_0)$. Then we can take $0<\rho_1<\rho_0$ such that for any $x\in M$ and any $\xi\in T_xM(\rho_1)$ and any $v\in\mathbb{R}^d$ with $\|v\|<\rho_1$, $\Phi_{v}(\exp_x(\xi))\in B_{\rho_0}(x)$. Then we can define a local flow $\tilde{\Phi}: O_{\rho_1}\times T_xM(\rho_1)\to T_xM$ as
$$\tilde{\Phi}(v, \xi)=\exp_x^{-1}\Phi(v, \exp_x\xi),$$
for any $(v,\xi)\in O_{\rho_1}\times T_xM(\rho_1)$. Since $D_x\Phi_0=Id$ for any $x\in M$, we can choose $\rho_1$ small enough such that $$\|D_\xi\tilde{\Phi}_v-Id\|<\frac{1}{4}$$
for any $x\in M, \xi\in T_xM(\rho_1)$ and $v\in O_{\rho_1}$.

Note that $\{X_1(x), X_2(x), \cdots, X_d(x)\}$ is linear independent. Also by the compactness of $M$ and continuity of $X_i(x)(1\leq i\leq d)$,  there is $C\geq 1$ such that for any $x\in M, (a_1, a_2, \cdots, a_d)\in\mathbb{R}^d$, one has
$$C^{-1}\sum_{i=1}^d|a_i|\leq\|a_1X_1(x)+a_2X_2(x)+\cdots+a_dX_d(x)\|\leq C\sum_{i=1}^d|a_i|.$$

Let $\tilde X_i(\xi)=(D_{\xi}\exp_{x})^{-1}(X_i(\exp_x\xi))$ for any $\xi\in T_xM(\rho_1)$ and $1\leq i\leq d$. Then $\tilde{X}_i$ is the vector field on $T_xM(\rho_1)$ which has solution curve $\tilde{\Phi}(te_i, \cdot)$ for any $1\leq i\leq d$. By choosing $\rho_1>0$ small enough, we can assume that
$$\|\tilde{X}_{i}(\xi)-X_i(x)\|<\frac{1}{4C}$$
for any $x\in M$ and $\xi\in T_xM(\rho_1)$.

Now we can choose $C^{-1}\rho_1>r_0>0$ small enough such that for any $\xi\in T_x M(r_0)$ and any $(a_1, a_2, \cdots, a_d)\in\mathbb{R}^d$ with $\|a_1X_1(x)+a_2X_2(x)+\cdots+a_dX_d(x)\|<r_0$ (note here we have $\|(a_1, a_2, \cdots, a_d)\|<\rho_1$), one has $\tilde\Phi((a_1, a_2, \cdots, a_d), \xi)\in T_x M(\rho_1)$.

Denote by $\tilde F_x=\exp_x^{-1}\circ F_x$. Then $\tilde{F}_x: T_xM(r_0)\to T_x M$ is a $C^1$ map defined on $T_xM(r_0)$. For any $\xi=\zeta+a_1X_1(x)+a_2X_2(x)+\cdots+a_dX_d(x)\in T_x M(r_0)$, a straightforward computation of the directional derivative of $\tilde{F}_x$ at $\xi$ along
the direction $X_i(x)$ gives
\begin{eqnarray*}
  D_\xi \tilde{F}_x(X_i(x)) &=& (\exp_x^{-1})_*(D_\xi F_x(X_i(x))) \\
   &=& (\exp_x^{-1})_*(X_i(\Phi((a_1,a_2, \cdots,a_d), \exp_x(\zeta)))) \\
   &=& \tilde{X}_i(\tilde{\Phi}((a_1, a_2, \cdots, a_d), \zeta))
\end{eqnarray*}
for any $1\leq i\leq d$. Thus for any $x\in M, \xi\in T_xM(r_0)$ and any $1\leq i\leq d$ we have
$$\|D_\xi \tilde{F}_x(X_i(x))-X_i(x)\|=\|\tilde{X}_i(\tilde{\Phi}((a_1, a_2, \cdots, a_d), \zeta))-X_i(x)\|<\frac{1}{4C}.$$
Then for any $\xi\in T_xM(r_0)$ and any $a_1X_1(x)+\cdots+a_dX_d(x)\in V_x$, we have
\begin{eqnarray*}
 &&\|D_\xi \tilde{F}_x(a_1X_1(x)+\cdots+a_dX_d(x))-(a_1X_1(x)+\cdots+a_dX_d(x))\|  \\
 && <\frac{1}{4C}\sum_{i=1}^d|a_i|\leq\frac{1}{2}\|a_1X_1(x)+\cdots+a_dX_d(x)\|.
\end{eqnarray*}
Hence we have $\|(D_\xi \tilde{F}_x-Id)|_{V_x}\|<\frac{1}{4}$ for any $x\in M, \xi\in T_xM(r_0)$.

Likewise, for any vector $\nu\in N_x$, a straightforward computation of the directional
derivative of $\tilde{F}_x$ at $\xi$ along the direction $\nu$ gives
$$D_\xi\tilde{F}_x\nu=D_\xi\tilde\Phi_{(a_1, a_2, \cdots, a_d)}\nu.$$
By the choice of $r_0<C^{-1}\rho_1$ we know that $\|(a_1, a_2, \cdots, a_d)\|\leq \rho_1$, and then
$$\|(D_\xi\tilde{F}_x-Id)|_{N_x}\|=\|(D_\xi\tilde\Phi_{(a_1, a_2, \cdots, a_d)}-Id)|_{N_x}\|<\frac{1}{4}.$$
For any $x\in M$ and $\xi\in T_xM(r_0)$, we have
$$\|D_\xi\tilde{F}_x-Id\|\leq \|(D_\xi\tilde{F}_x-Id)|_{N_x}\|+\|(D_\xi\tilde{F}_x-Id)|_{V_x}\|<\frac{1}{4}+\frac{1}{4}=\frac{1}{2}.$$

For any $\xi\in T_xM(r_0)$, by the fact that $\|D_\xi\tilde{F}_x-Id\|<\frac{1}{2}$ we know that $D_\xi\tilde{F}_x$ is a linear isomorphism, then by the inverse function theorem we know that $\tilde{F}_x:T_x M(r_0)\to T_x M$ is a local diffeomorphism. Now we proceed to prove that $\tilde{F}_x:T_x M(r_0)\to T_x M$ is injective. Denote by $g=F_x-Id: T_x M(r_0)\to T_xM$, then $g$ is Lipschitz map with Lipschitz constant $\frac{1}{2}$ since $\|D_\xi g\|<\frac{1}{2}$ for any $\xi\in T_xM(r_0)$. For any $\xi\in T_xM$, we prove that there exists at most one $\xi'\in T_x M(r_0)$ such that $\tilde{F}_x(\xi')=\xi'+g(\xi')=\xi$. It is easy to see that $\tilde{F}_x(\xi')=\xi'+g(\xi')=\xi$ is equivalent to $\xi'=\xi-g(\xi')$. Fixed $\xi$, consider $T: T_xM(r_0)\to T_x M$ defined by $T(\xi')=\xi-g(\xi')$. One can check that $\|T(\xi')-T(\xi'')\|=\|g(\xi')-g(\xi'')\|<\frac{1}{2}\|\xi'-\xi''\|$ for any $\xi',\xi''\in T_xM(r_0)$. Thus there exists at most one point $\xi'\in T_xM(r_0)$ such that $T(\xi')=\xi'$, that is, there exists at most one point $\xi'$ such that $\tilde{F}_x(\xi')=\xi$. This proves that $\tilde{F}_x:T_xM(r_0)\to T_x M$ is injective. This proves that $\tilde{F}_x$ is an embedding from $T_xM(r_0)$ to $T_xM$ and then $F_x=\exp_x\circ\tilde{F}_x$ is an embedding from $T_x M(r_0)$ to $M$.

For any $\xi\in T_xM(r_0)$, we have $$\|D_{\xi}\tilde{F}_x\|\leq \|Id\|+\|D_{\xi}\tilde{F}_x-Id\|<\frac{3}{2} \text{ and } m(D_{\xi}\tilde{F}_x)>m(Id)-\|D_{\xi}\tilde{F}_x-Id\|>\frac{1}{2}.$$ Then we can easily check that
$$\|D_{\xi}F_x\|=\|D_\xi(\exp_x\circ \tilde{F}_x)\|<\frac{3}{2}\times\frac{3}{2}<3, \text{ and } m(D_{\xi}F_x)=m(D_\xi(\exp_x\circ \tilde{F}_x))>\frac{1}{2}\times\frac{2}{3}=\frac{1}{3}.$$
This ends the proof of the lemma.
\end{proof}

\begin{Lemma}\label{lem8}
Let $M$ be a compact Riemannian manifold without boundary, $\Phi:\mathbb{R}^d\times M\to M$ be a $C^1$ homogeneous $\mathbb{R}^d$-action, $0<d\leq\dim M$. For any $\mu\in(0, \varepsilon_0(\Phi)/3)$, there is $\delta>0$ small enough such that for any $x\in M$, if a continuous map $\alpha:\bar{O}_1\subset\mathbb{R}^d\to B_{\delta}(x)$ satisfied $\alpha(v)\in Orb(x, \Phi)$ and $\alpha(0)=x$, then $\alpha(v)\in\Phi_{\bar{O}_{\mu}}(x)$ for any $v\in\bar{O}_1$.
\end{Lemma}

\begin{proof}
Take $r_0$ as above Lemma \ref{lem7}. Let $\mu\in(0, \varepsilon_0(\Phi)/3)$ be given. Choose $0<\delta<r_0/3$ such that for any $(a_1, a_2, \cdots, a_d)\in \mathbb{R}^d$,
$$\|a_1X_1(x)+a_2X_2(x)+\cdots +a_dX_d(x)\|<3\delta$$
implies $(a_1, a_2, \cdots, a_d)\in \bar{O}_\mu$ for any $x\in M$.

For any $x\in M$, denote by $\pi_x$ the orthogonal projection from $T_xM$ to $N_x$. Note that $m(D_\xi F_x)>1/3$ for any $\xi\in T_xM(r_0)$ and $F_x(0_x)=x$, we have $B_\delta(x)\subset F_x(T_xM(r_0))$. Thus we can define a continuous map $P_x: B_{\delta}(x)\to N_x$ by letting $P_x(y)=\pi_x\circ F_x^{-1}(y)$ for any $y\in B_\delta(x)$.

Let $\alpha:\bar{O}_1\subset\mathbb{R}^d\to B_{\delta}(x)$ be a continuous map with $\alpha(v)\in Orb(x, \Phi)$ for all $v\in \bar{O}_1$ and $\alpha(0)=x$. We can get a continuous map $P_x\circ\alpha$ with image in $N_x$. Given $u\in\mathbb{R}^d$ with $\Phi_u(x)\in B_\delta(x)$, we can write
$$F_x^{-1}(\Phi_u(x))=\zeta+a_1X_1(x)+a_2X_2(x)+\cdots +a_dX_d(x),$$
where $\zeta=P_x(\Phi_u(x))$ and $(a_1,a_2,\cdots,a_d)\in\mathbb{R}^d$. We can find rational point $$u'=(a_1', a_2', \cdots, a_d')\in\mathbb{R}^d$$ arbitrarily close to $u$ such that $F_x(\zeta+a_1'X_1(x)+a_2'X_2(x)+\cdots+a_d'X_d(x))\in B_\delta(x)$, then we know that $P_x(\Phi_u(x))=P_x(\Phi_{u'}(x))=\zeta$. From this fact we can see that the image of
$$\{\Phi_u(x)|u\in\mathbb{R}^d, \Phi_{u}(x)\in B_{\delta}(x)\}$$
under the $P_x$, is at mostly countable. Thus the image of $P_x\circ \alpha$ consists of at mostly countable points in $N_x$.

By definition we can see that $F_x(0_x)=x$, thus $P_x(x)=\pi_x(F_x^{-1}(x))=0_x$. Assume that we have a point $v\in\bar{O}_1$ such that $P_x\alpha(v)\neq 0_x$. We can take a line segment $\gamma(t)$ in $\bar{O}_1$ connecting $0_x$ and $v$. Thus we know that $P_x\circ\alpha\circ\gamma$ is a curve in $N_x$, this contradicts with the image of $P_x\circ\alpha$ is at mostly countable. Hence we have $P_x\circ\alpha(v)=0_x$ for any $v\in\bar{O}_1$. For any $v\in\bar{O}_1$, we know that $P_x(\alpha(v))=0_x$, then we know that $F_x^{-1}(\alpha(v))=a_1X_1(x)+a_2X_2(x)+\cdots+a_dX_d(x)$ for some $(a_1, a_2, \cdots, a_d)\in\mathbb{R}^d$. Note that $$\|a_1X_1(x)+a_2X_2(x)+\cdots+a_dX_d(x)\|=\|F_x^{-1}(\alpha(v))\|<3\delta,$$ thus we have $(a_1, a_2, \cdots, a_d)\in \bar{O}_\mu$ by the choice of $\delta$. And we have $\alpha(v)=\Phi_{(a_1, a_2, \cdots, a_d)}(x)$ by the definition of $F_x$. This ends the proof of the lemma.
\end{proof}

\begin{Lemma}\label{lem9}
Let $\Phi:\mathbb{R}^d\times M\to M$ be a separating $C^1$ homogeneous $\mathbb{R}^d$-action on a compact boundaryless Riemannian manifold $M$. For any $\Psi\in\mathcal{Z}^0(\Phi)$, there exist constant $a>0,\mu\in(0, \varepsilon_0(\Phi)/3)$ and a map $z:\bar{O}_a\times M\to\bar{O}_\mu$ such that $\Psi_v(x)=\Phi_{z(v, x)}(x)$ for any $(v,x)\in\bar{O}_a\times M$. Moreover,
\begin{enumerate}
\item $z$ is continuous;
\item $z(u+v,x)=z(u,x)+z(v,\Psi_u(x))$ for any $x\in M$ and $u,v\in\bar{O}_a$ with $v+u\in\bar{O}_a$;
\item $z(u, \Phi_v(x))=z(u, x)$ for any $x\in M$ and $u\in\bar{O}_a$ and $v\in\mathbb{R}^d$;
\item there exists a continuous map $A: M\to\mathcal{M}_{d\times d}(\mathbb{R})$ such that $z(u, x)=A(x)u$ for any $x\in M$ and $u\in\bar{O}_a$.
\end{enumerate}
\end{Lemma}

\begin{proof}
Let $\mu\in(0, \varepsilon_0(\Phi)/3)$ be given and $\delta$ be the constant as in Lemma \ref{lem8}. Since $\Phi$ is separating, we can take $\delta>0$ small enough such that for any $x,y\in M$, if
$$d(\Phi_v(x), \Phi_v(y))<\delta$$
for all $v\in\mathbb{R}^d$, then $y\in Orb(x,\Phi)$. $\delta$ is the separating
constant of $\Phi$. Because $\Psi$ is continuous on $\mathbb{R}^d\times M$, according to compactness of $M$, there exists a constant $a>0$ such that
$$d(\Psi_u(x),x)<\delta$$
for all $(u,x)\in\bar{O}_{a}\times M$.

For all $u\in\bar{O}_a,v\in\mathbb{R}^d, x\in M$, we have $$d(\Phi_u(x),\Phi_v(\Psi_u(x)))=d(\Phi_v(x),\Psi_u(\Phi_v(x)))<\delta.$$
Since $\delta$ is the separating constant, we have $\Psi_u(x)\in Orb(x, \Phi)$ for all $u\in\bar{O}_a$. Note that $\Psi_u(x)\in B_{\delta}(x)$ for all $u\in\bar{O}_{a}$ and $\Psi_0(x)=x$, by Lemma \ref{lem8} we know that there is $\eta=z(u,x)\in\bar{O}_{\mu}$ such that $\Psi_u(x)=\Phi_{\eta}(x)$. By the choice of $\mu\in(0, \varepsilon_0(\Phi)/3)$, we know that for any $u_1, u_2\in\bar{O}_{\mu}$, when $u_1\neq u_2$ we have $\Phi_{u_1}(x)\neq\Phi_{u_2}(x)$ for any $x\in M$. Hence $\eta=z(u,x)$ is uniquely defined on $(u,x)\in\bar{O}_a\times M$. This gives a map $\eta=z(u,x)$ on $\bar{O}_a\times M$.

If $\eta=z(u,x)$ is not continuous, then one can find a sequence of $\{(u_n, x_n)\}$ in $\bar{O}_a\times M$ with $(u_n, x_n)\to (u_0, x_0)\in\bar{O}_a\times M$ as $n\to\infty$ such that $\|z(u_n, x_n)-z(u_0, x_0)\|\nrightarrow 0$. By choosing a subsequence we can assume that $z(u_n, x_n)-z(u_0, x_0)\to \eta_0\in\bar{O}_{\mu}$. Since $\eta_0\neq 0$, we have
$$\lim\limits_{n\to\infty}d(\Phi_{z(u_n, x_n)}(x_0), \Phi_{z(u_0, x_0)}(x_0))=d(\Phi_{z(u_0, x_0)+\eta_0}(x_0), \Phi_{z(u_0, x_0)}(x_0))\neq 0.$$
On the other hand we have
\begin{eqnarray*}
 d(\Phi_{z(u_n, x_n)}(x_0), \Phi_{z(u_0, x_0)}(x_0)) &\leq& d(\Phi_{z(u_n, x_n)}(x_n), \Phi_{z(u_n, x_n)}(x_0))+d(\Phi_{z(u_n, x_n)}(x_n), \Phi_{z(u_0, x_0)}(x_0)) \\
   &=& d(\Phi_{z(u_n, x_n)}(x_n), \Phi_{z(u_n, x_n)}(x_0))+d(\Psi_{u_n}(x_n), \Psi_{u_0}(x_0)).
\end{eqnarray*}
Since $\|z(u_n, x_n)\|$ is bounded and $d(x_n, x_0)\to 0$ we know $d(\Phi_{z(u_n, x_n)}(x_n), \Phi_{z(u_n, x_n)}(x_0))\to 0$. Since $u_n\to u_0$ and $x_n\to x_0$ we have $d(\Psi_{u_n}(x_n), \Psi_{u_0}(x_0))\to 0$, thus we have
$$\lim\limits_{n\to\infty}d(\Phi_{z(u_n, x_n)}(x_0), \Phi_{z(u_0, x_0)}(x_0))=0,$$ a contradiction. This proves that $\eta=z(u,x)$ is continuous on $\bar{O}_a\times M$.

Let $u, v\in\bar{O}_a$ with $u+v\in\bar{O}_a$ be given. For any $x\in M$, we have
\begin{eqnarray*}
  \Phi_{z(u+v,x)}(x) &=& \Psi_{u+v}(x)=\Psi_v(\Psi_u(x))=\Phi_{z(v, \Psi_u(x))}(\Psi_u(x)) \\
   &=& \Phi_{z(v, \Psi_u(x))}(\Phi_{z(u,x)}(x))=\Phi_{z(u,x)+z(v,\Psi_u(x))}(x),
\end{eqnarray*}
and then
$$\Phi_{z(u+v,x)-(z(u,x)+z(v,\Psi_u(x)))}(x)=x.$$
Note that $\|z(u+v,x)-(z(u,x)+z(v,\Psi_u(x)))\|\leq 3\mu<\varepsilon_0(\Phi)$, then we can see that $z(u+v,x)-(z(u,x)+z(v,\Psi_u(x)))=0$ by the property of $\varepsilon_0(\Phi)$. Thus for any $x\in M$ and $u,v\in\bar{O}_a$ with $u+v\in\bar{O}_a$, we have $z(u+v,x)=z(u,x)+z(v,\Psi_u(x))$. This proves item 2 of the lemma.

Let $v\in\bar{O}_a$, $u\in\bar{O}_{\mu}$ and $x\in M$ be given. Note that
\begin{eqnarray*}
  \Phi_{u+z(v, \Phi_u(x))}(x) &=& \Phi_{z(v, \Phi_u(x))}(\Phi_u(x))=\Psi_v(\Phi_u(x)) \\
   &=& \Phi_u(\Psi_v(x))=\Phi_u(\Phi_{z(v,x)}(x))=\Phi_{u+z(v, x)}(x).
\end{eqnarray*}
Thus we have $\Phi_{z(v, \Phi_u(x))-z(v, x)}(x)=x$. Since $\|z(v, \Phi_u(x))-z(v, x)\|\leq 2\mu<\varepsilon_0(\Phi)$, we have $z(v, \Phi_u(x))-z(v, x)=0$. This proves that $z(v, \Phi_u(x))=z(v, x)$ holds for any $x\in M, v\in\bar{O}_a$ and $u\in\bar{O}_\mu$.

Let $v\in\bar{O}_a$ and $u\in\mathbb{R}^d$ and $x\in M$ be given. We can find $n\in\mathbb{N}$ big enough such that $\|n^{-1}u\|\leq\mu$, then we can see that
$$z(v,\Phi_{\frac{1}{n}u}(x))=z(v,x).$$
For $\Phi_{\frac{1}{n}u}(x)$, we have
$$z(v,\Phi_{\frac{2}{n}u}(x))=z(v,\Phi_{\frac{1}{n}u}(\Phi_{\frac{1}{n}u}(x)))=z(v,\Phi_{\frac{1}{n}u}(x))=z(v,x).$$
Inductively, we have
$$z(v, \Phi_u(x))=z(v, \Phi_{\frac{n-1}{n}u}(x))=\cdots=z(v, \Phi_{\frac{1}{n}u}(x))=z(v,x).$$
This proves item 3 of the lemma.

From item 2 and 3 we can see that for any $v, u\in\bar{O}_a$ with $u+v\in\bar{O}_a$ and any $x\in M$, we have
$$z(v+u, x)=z(u, x)+z(v, \Psi_u(x))=z(u, x)+z(v, \Phi_{z(u, x)}(x))=z(u, x)+z(v, x).$$
Then we can see that for any $n\in\mathbb{Z}^+,u\in\mathbb{R}^d$ with $\|u\|\leq a$, we have $nz(n^{-1}u, x)=z(u, x)$, and then we have $z(n^{-1}u, x)=n^{-1}z(u, x)$. And then we have $z(\frac{m}{n}u, x)=\frac{m}{n}z(u,x)$ for any rational number $\frac{m}{n}\in[0, 1]$. Note that $z(-u, x)=-z(u, x)$, we can see that $z(\frac{m}{n}u, x)=\frac{m}{n}z(u,x)$ for any rational number $\frac{m}{n}\in[-1,1]$, by the continuity of $z(u, x)$ we can see that $z(tu, x)=tz(u,x)$ for any $t\in[-1,1]$.

Assume that $$a^{-1}z(ae_i, x)=\sum_{j=1}^d a_{ji}(x)e_j.$$
We get a matrix $A(x)=(a_{ji}(x))_{d\times d}$. For any $u=t_1e_1+t_2e_2+\cdots+t_de_d\in\mathbb{R}^d$ with $\|u\|\leq a$, we have
$$z(u, x)=\sum_{i=1}^d z(t_ie_i, x)=\sum_{i=1}^da^{-1}t_iz(a e_i, a)=\sum_{i=1}^{d}\sum_{j=1}^da_{ji}(x)t_ie_j=A(x)u.$$
This proves item 4 of the lemma.
\end{proof}

Now we proceed to prove Theorem B.

\bigskip

{\noindent\it Proof of Theorem B.}
Let $\Psi\in\mathcal{Z}^0(\Phi)$. Then we can take $a$ and $z(u, x)=A(x)u$ as in Lemma \ref{lem9}.
By the continuity of $z(u, x)$ and the expression of $A(x)$ we can see that $A(x)$ is continuous on $M$. Let $x\in M$ be given, for any $u\in\mathbb{R}^d$ with $\|u\|\leq a$, we have
$$A(\Phi_v(x))u=z(u,\Phi_v(x))=z(u,x)=A(x)u$$
for any $v\in\mathbb{R}^d$. By the expression of $A(x)$, one can get that $A(\Phi_v(x))=A(x)$ for any $x\in M$ and $v\in\mathbb{R}^d$. Hence $A(x)$ is constant along orbit of $\Phi$.

By the fact that $\Psi_u(x)=\Phi_{z(u,x)}(x)$ for any $x\in M$ and $u\in\bar{O}_a$, we can easily see that $\Psi_u(x)=\Phi_{A(x)u}(x)$ is true for any $x\in M$ and $u\in\bar{O}_a$. Fixed any $x\in M$ and $u\in\mathbb{R}^d$, we can take $n\in\mathbb{N}$ big enough such that $\|n^{-1}u\|\leq a$. Denote by $\tau=n^{-1}u$, then we have
$$\Psi_{2\tau}(x)=\Psi_\tau(\Psi_\tau(x))=\Phi_{A(\Psi_{\tau}(x))\tau}(\Phi_{A(x)\tau}(x))=\Phi_{A(\Phi_{A(x)\tau}(x))\tau}(\Phi_{A(x)\tau}(x))=\Phi_{2A(x)\tau}(x),$$
$$\Psi_{3\tau}(x)=\Psi_\tau(\Psi_{2\tau}(x))=\Phi_{A(\Psi_{2\tau}(x))\tau}(\Phi_{2A(x)\tau}(x))=\Phi_{A(\Phi_{2A(x)\tau}(x))\tau}(\Phi_{2A(x)\tau}(x))=\Phi_{3A(x)\tau}(x),$$
and then by induction we can get that $\Psi_u(x)=\Psi_{n\tau}(x)=\Phi_{nA(x)\tau}(x)=\Phi_{A(x)u}(x)$. This ends the proof of Theorem B.
\hfill\qedsymbol

\begin{Remark}\label{cor2}
Consider a $C^1$-separating and homogeneous $\mathbb{R}^d$-action $\Phi:\mathbb{R}^d\times M\to M$,
and $\Psi\in\mathcal{Z}^1(\Phi)$. By Theorem B we know that there is $A:M\to\mathcal{M}_{d\times d}(\mathbb{R})$ such that $\Psi(v,x)=\Phi(A(x)v,x)$ for all $v\in\mathbb{R}^d$ and $x\in M$. Denote by $X_i(\cdot)=\frac{d\Phi(te_i,\cdot)}{dt}|_{t=0}$ and $Y_i(\cdot)=\frac{d\Psi(te_i,\cdot)}{dt}|_{t=0}$ the vector field generated by $\Phi$ and $\Psi$ respectively. The linear map $A(x)$ is represented by the matrix of representation of the vectors $(Y_i(x))_{1\leq i\leq d}$ on the basis $(X_i(x))_{1\leq i\leq d}$.
\end{Remark}


\begin{thebibliography}{1d}

\bibitem{A}
V. I. Arnold, \textit{Mathematical Methods of Classical Mechanics,} 2nd edn. Springer, New York (1991).

\bibitem{Ar2}  A. Artigue, \textit{Discrete and continuous topological dynamics: fields of cross sections and expansive flows.} Discrete Contin. Dyn. Syst. \textbf{36} (2016), no. 11, 5911--5927.

\bibitem{Ar}
A. Artigue, \textit{Rescaled expansivity and separating flows.} Discrete Contin. Dyn. Syst. \textbf{38} (2018), no. 9, 4433--4447.

\bibitem{BFH}
L. Bakker, T. Fisher, B. Hasselblatt, \textit{Centralizers of hyperbolic and kinematic-expansive flows.} Math. Res. Rep. \textbf{2} (2021), 21--44.

\bibitem{BRV}
W. Bonomo, J. Rocha, P. Varandas, \textit{The centralizer of Komuro-expansive flows and expansive $\mathbb{R}^d$-actions.} Math. Z. \textbf{289} (2018), no. 3-4, 1059--1088.

\bibitem{BV}
W. Bonomo, P. Varandas, \textit{A criterion for the triviality of the centralizer for vector fields and applications.} J. Differential Equations \textbf{267} (2019), no. 3, 1748--1766.

\bibitem{BW}
R. Bowen, P. Walters, \textit{Expansive one-parameter flows.} J. Differential Equations \textbf{12} (1972), 180--193.

\bibitem{G}
A. A. Gura, \textit{Horocycle flow on a surface of negative curvature is separating.} Mat. Zametki \textbf{36} (1984), no. 2, 279--284.

\bibitem{KH}
A. Katok, B. Hasselblatt, \textit{Introduction to the modern theory of dynamical systems.} Cambridge University Press, Cambridge, (1995).

\bibitem{KM}
K. Kato, A. Morimoto, \textit{Topological stability of anosov flows and their centralizers.} Topology \textbf{12} (1973), 255--273.

\bibitem{L}
J. M. Lee, \textit{Introduction to Smooth Manifolds,} 2nd edn. Springer, New York, (2013).

\bibitem{LOS}
M. Leguil, D. Obata, B. Santiago, \textit{On the centralizer of vector fields: criteria of triviality and genericity results.} Math. Z. \textbf{297} (2021), no. 1-2, 283--337.

\bibitem{O}
M. Oka. \textit{Expansive flows and their centralizers.} Nagoya Math. J. \textbf{64} (1976), 1--15.

\bibitem{PY}
J. Palis, J. C. Yoccoz, \textit{Rigidity of centralizers of diffeomorphisms.} Sci. \'Ecole Norm. Sup. (4) \textbf{22} (1989), no. 1, 81--98.

\bibitem{S}
P. R. Sad, \textit{Centralizers of vector fields.} Topology \textbf{18} (1979), no. 2, 97--104.

\bibitem{Wh}
H. Whitney, \textit{Regular families of curves.} Ann. of Math. (2) \textbf{34} (1933), no. 2, 244--270.

\end{thebibliography}
\end{document}